\documentclass{article}
\usepackage[top=3cm, bottom=2.5cm, left=2.5cm, right=2.5cm]{geometry} 
\usepackage{latexsym}
\usepackage{amsthm}
\usepackage{amssymb}
\usepackage{graphicx}
\newcommand{\Z}{\mathbb{Z}}

\newcommand{\R}{\mathbb{R}}

\def\Vec{\mathop{\rm Vec}\nolimits}
\newtheorem{theorem}{Theorem}

\newtheorem{proposition}{Proposition}
\newtheorem{lemma}{Lemma}

\title{Applications of parabolic Dirac operators to the instationary viscous MHD equations on conformally flat manifolds}
\author{
P. Cerejeiras \thanks{Departamento de Matem\'atica, Universidade de Aveiro, P 3810-193 Aveiro, Portugal. E-Mail: {\tt pceres@ua.pt}},
\and
U. K\"ahler \thanks{Departamento de Matem\'atica, Universidade de Aveiro, P 3810-193 Aveiro, Portugal. E-Mail: {\tt ukaehler@ua.pt}},
\and
R.S.~Krau{\ss}har \thanks{Fachgebiet Mathematik, Erziehungswissenschaftliche Fakult\"at, Universt\"at Erfurt, Nordh\"auser Str. 63, D-99089 Erfurt, Germany.  E-mail: {\tt soeren.krausshar@uni-erfurt.de}}}
\begin{document}
\maketitle
\begin{abstract}
In this paper we apply classical and recent techniques from quaternionic analysis using parabolic Dirac type operators and related Teodorescu and Cauchy-Bitzadse type operators to set up some analytic representation formulas for the solutions to the time depedendent incompressible viscous magnetohydrodynamic  equations on some conformally flat manifolds, such as cylinders and tori associated with different spinor bundles. Also in this context a special variant of hypercomplex Eisenstein series related to the parabolic Dirac operator serve as kernel functions.

\end{abstract}

\par\medskip\par
In honor of Professor Spr\"o{\ss}ig's 70th birthday
\section{Introduction}

The magnetohydrodynamic equations (MHD) represent a combination of the Navier-Stokes system with the Maxwell system. They describe fluid dynamical processes under the influence of an electromagnetic field {and have been the subject of investigation of }
numerous authors since more than twenty years. {As classical references we emphasize }
~\cite{ST} among others. 

{In general, there is a distinction made between the }
 inviscid and the viscous MHD equations. On the one hand, the inviscid MHD equations play an important role in the description of the dynamic of astrophysical plasmas, for instance in the description of the magnetic phenomena of the heliosphere and in the prediction of the distribution of the solar wind density, see for example \cite{GP} and the references therein. {On the other hand,  the viscous MHD equations have attracted a growing interest by mathematicians and physicists over the last three decades. This topic is in the main focus of recent interest, see for instance \cite{BD,GeSh2015,Lei2015,XZ-ZY2017}, where new criteria concerning the existence of global solutions and global well-posedness for particular geometrical settings, in particular axially symmetric settings are being developed. Also, it has recently been applied to medicine, such as in modelling of hydromagnetic blood flows \cite{REM2017}. More classical results can be found in \cite{GMP}.}

\par\medskip\par
In this paper we revisit the three dimensional instationary incompressible viscous MHD equations  
\begin{eqnarray}
-\frac{1}{Re}\Delta {\bf u} + \frac{\partial {\bf u}}{\partial t} + ({\bf u} \;{\rm grad})\; {\bf u} + {\rm grad}\; p  &=& \frac{1}{\mu_0} {\rm rot} {\bf B} \times {\bf B} \;\;{\rm in}\;G\\
-\frac{1}{Rm}\Delta {\bf B} + \frac{\partial {\bf B}}{\partial t} + ({\bf u} \;{\rm grad})\; {\bf B} - ({\bf B} \;{\rm grad}) {\bf u} &=& 0\;\;{\rm in}\; G \\
{\rm div}\; {\bf u} &=& 0 \;\;{\rm in}\;G\\
{\rm div}\; {\bf B} &=& 0 \;\;{\rm in}\;G\\
{\bf u}={\bf 0},\; {\bf B} &=& {\bf h} \;\;{\rm at}\;\partial G.
\end{eqnarray}
In the context of this paper $G$ is some arbitrary time-varying Lipschitz domain $G \subset \R^3 \times \R^+$. The symbol ${\bf u}$ represents the velocity of the flow, $p$ the pressure, ${\bf B}$ the magnetic field, $\mu_0$ is magnetic permeability of the vacuum and $Re$ and $Rm$ the fluid mechanical resp. magnetic Reynolds number. The first equation basically resembles the time dependent Navier-Stokes equation - the external force however is an unknown magnetic entity that also needs to be computed. Together with the second equation the dynamics of the magnetic field, the velocity, and the pressure, is described. The third equation manifests the incompressibility of the flow. The forth equation states the non-existence of magnetic monopoles. The remaining equations  represent the measured (known) data at the boundary $\Gamma=\partial G$ of the domain $G$. 
\par\medskip\par

In \cite{CMZ,Gala,MY} some global existence criteria for the weak solutions to the instationary 3D MHD equations have been presented. These works use modern harmonic analysis techniques as proposed in \cite{Cannone} for the incompressible Navier-Stokes equations. However, many theoretical questions concerning existence, uniqueness and regularity in the framework of general domains still remain open problems. In particular, one is interested in improving the explicitness of these criteria and in obtaining explicit analytic representation formulas for the solutions as well as for the Lipschitz contraction constant being valid in all kinds of Lipschitz domains --- independently of the particular geometry of the domain. 

Furthermore, we observed that in many cases dealing with large temporal distances, the classical time stepping methods {(like the Rothe method) are valid for only small periods of time and, therefore, they }
 often do not lead to the desired result. These obstacles motivate us to develop alternative methods. 

\par\medskip\par

Over the last three decades the quaternionic operator calculus proposed by K. G\"urlebeck, W. Spr\"o{\ss}ig, M. Shapiro, V.V. Kravchenko, P. Cerejeiras, U. K\"ahler and by their collaborators, see for example \cite{CK1,CK2,GS1,K}, provides an alternative analytic toolkit to treat the Navier-Stokes system, the Maxwell system and many other elliptic PDE. The quaternionic calculus leads  to further new explicit criteria for the regularity, the existence and the uniqueness of the solutions. Moreover, it turned out to be also suitable to tackle strongly time dependent problems very elegantly. Based on the new theoretical results also new numerical algorithms could be developed, see for instance \cite{FGHK}. Also fully analytic representation formulas for the solutions to the Navier-Stokes equations and for the Maxwell and Helmholtz systems could be established for some special classes of domains, cf. \cite{ConKra5,CK2009}. An important advantage of the quaternionic calculus is that the formulas hold universally for all bounded Lipschitz domains, independently of its particular geometry.

As shown already by Sijue Wu in \cite{SW}, quaternionic analytic methods could also be applied to deal the well posedness problem in Sobolev spaces of the full 3D water wave problem, where previously well established methods did not lead to any  success.

\par\medskip\par

Since the quaternionic calculus provided an added value both in the treatment of the Navier-Stokes system and of the Maxwell system, it is natural to expect similar insightful results for the MHD system, since the latter one is a coupling of both systems. In \cite{KraussharTrends1} we explained how we can compute the solutions of the time independent stationary incompressible viscous MHD system with the quaternionic integral operator calculus. Recently complex quaternions have also been used in \cite{TDT} to describe the dynamics of dyonic plasmas in an elegante way. In future work we plan to address the fully time-dependent incompressible viscous MHD equations using parabolic versions of the Dirac operator for modelling these type of equations independent of particular geometric constraints - except of regularity conditions on the boundaries

\par\medskip\par 

The aim of this paper is to exploit another advantage of quaternionic methods - namely that they are naturally predestinated to also address analogous MHD problems in the more general context of conformally flat spin manifolds that arise by factoring out some simply connected domain by a discrete Kleinian group. In this paper we specifically look at MHD problems on several kinds of conformally flat spin cylinders and tori as these are the most illustrative examples. In particular, this paper provides a generalization of the idea used in \cite{CKK2018} were we addressed the ``simpler'' Navier-Stokes equations on these kind of manifolds without the influence of a magnetic field.

It is worth to mention that in the same way how we treat flat spin cylinders or tori we can also address  their non-oriented conformally flat twisted analogues - namely the M\"obius strip and the Kleinian bottle - where we have pin-  instead of spin-structures.   

The construction methods can easily be adapted by replacing the corresponding  integral kernels. In this paper we explain how to explicit construct the integral kernels and how these are used in the resolution schemes for our specific MHD problem on the cylinders tori.  We finalize with a brief look at particular rotation-invariant variants of these varieties and explain how our construction can easily be transferred to this setting.

\section{Preliminaries}

\subsection{The quaternionic operator calculus}

By $e_1,e_2,e_3$ we denote the usual vector space basis $\mathbb{R}^3$. To introduce a multiplication operation on $\mathbb{R}^3$, we embed it into the
algebra of Hamiltonian quaternions $\mathbb{H}$. A
quaternion has the form $x=x_0+{\bf x} := x_0 + x_1
e_1 + x_2 e_2 + x_3 e_3$ where $x_0,\ldots,x_3$ are real numbers. Furtmore,
$x_0$ is called the real part of the quaternion and will be
denoted by $\Re(x)$. ${\bf x}$ is the vector part of $x$, also denoted by $\Vec(x)$. 
In the quaternionic setting the standard basis 
vectors play the role of imaginary units, we have $e_i^2 = -1$ for
$i=1,2,3$. Their mutual multiplication coincides with the usual
vector product, i.e., $e_1 e_2 =  e_3, e_2 e_3 = e_1$, $e_3 e_1 =
e_2$ and $e_i e_j = - e_j e_i$ for $i \neq j$. We also need the quaternionic 
conjugation defined by
$\overline{ab} = \overline{b} \; \overline{a},\; \overline{e_i} =
- e_i$, $i=1,2,3$. The usual Euclidean norm extends to a norm on the
whole quaternionic algebra, i.e. $|a|:= \sqrt{\sum_{i=0}^3
a_i^2}$.

\par\medskip\par

The additional multiplicative structure of the quaternions allows
us to describe all $C^1$-functions $ {\bf f}:\mathbb{R}^3 \rightarrow
\mathbb{R}^3$ that satisfy both ${\rm div}\;{\bf f} = 0$ and ${\rm
rot}\;{\bf f}=0$ equivalently in a compact form as null-solutions to one 
single differential operator. The latter is the three-dimensional Euclidean Dirac
operator ${\bf D}:=\sum_{i=1}^3 \frac{\partial }{\partial x_i}
e_i$. In spin geometry this operator is also known as the Atiyah-Singer-Dirac operator. It naturally arises from the Levi-Civita connection in the context of general Riemannian spin manifolds, reducing to the above stated simple form in the flat case. In turn, 
the Euclidean Dirac operator coincides with the
usual gradient operator when this one is applied to a
scalar-valued function. If $U \subseteq \mathbb{R}^3$ is an open
subset, then a real differentiable function $f: U \rightarrow
\mathbb{H}$ is called left quaternionic holomorphic or left
monogenic in $U$, if ${\bf D}f = 0$. In the quaternionic calculus,
the square of the Euclidean Dirac operator gives the Euclidean
Laplacian up to a minus sign; we have ${\bf D}^2 = - \Delta$. Consequently,
every real component of a left monogenic function is harmonic. This property
allows us to treat harmonic functions with the function theory of
the Dirac operator offering generalizations of many powerful
theorems used in complex analysis. For deeper insight, we refer the reader for
instance to \cite{DSS,GS1}. 

\par\medskip\par

To treat time dependent problems in $\R^3$ we follow the ideas of \cite{CK2} and introduce the ``parabolic'' basis elements $\mathfrak{f}$ and $\mathfrak{f}^{\dagger}$ which act in the following way 
\begin{eqnarray*}
\mathfrak{f} \mathfrak{f}^{\dagger} + \mathfrak{f}^{\dagger} \mathfrak{f} &=& 1,\\
\mathfrak{f}^2 = (\mathfrak{f}^{\dagger})^2 &=& 0,\\
\mathfrak{f} e_j = e_j \mathfrak{f} &=& 0,\\
\mathfrak{f}^{\dagger} e_j = e_j \mathfrak{f}^{\dagger} &=& 0.
\end{eqnarray*}
The associated parabolic Dirac operators have the form 
$$
D_{{\bf x},t}^{\pm} := \sum\limits_{j=1}^3 e_j \frac{\partial }{\partial x_j} + \mathfrak{f} \frac{\partial }{\partial t} \pm \mathfrak{f}^{\dagger}
$$
and satisfy $(D_{{\bf x},t}^{\pm})^2 = -\Delta \pm \frac{\partial }{\partial t}$. The fundamental solution to $D_{{\bf x},t}^{+}$ has the form 
$$
G({\bf x},t) = \frac{H(t) \exp(-\frac{|{\bf x}|^2}{4t})}
{(2 \sqrt{\pi t})^3}
\Big(\frac{1}{2t} \sum\limits_{j=1}^3 e_j x_j + \mathfrak{f}(\frac{3}{2t} 
+ \frac{|{\bf x}|^2}{4t^2})+ \mathfrak{f}^{\dagger}\Big),
$$
where $H(\cdot)$ stands for the usual Heaviside function. 
Solutions satisfying $D_{{\bf x},t}^{\pm} f = 0$ are called left parabolic monogenic (resp. antimonogenic). 
\par\medskip\par
For our needs we need the more general parabolic Dirac type operator, used for instance in \cite{Be,CV2009}, having the form 
$$
D_{{\bf x},t,k}^{\pm} := \sum\limits_{j=1}^3 e_j \frac{\partial }{\partial x_j} + \mathfrak{f} \frac{\partial }{\partial t} \pm k \mathfrak{f}^{\dagger}
$$ 
for a positive real $k \in \R$. This operator factorizes the second order operator 
$$
(D_{{\bf x},t,k}^{\pm})^2 = -\Delta \pm k^2 \frac{\partial }{\partial t}
$$
and has very similar properties as the previously introduced one. Its nullsolutions are called left parabolic $k$-monogenic (resp. left parabolic $k$-antimonogenic) functions. 
\par\medskip\par
Adapting from \cite{Be,CV2009}, the fundamental solution to $D_{{\bf x},t,k}^{+}$ turns out to have the form 
$$
E({\bf x},t;k) = \sqrt{k}~\frac{H(t) \exp(-\frac{k|{\bf x}|^2}{4t})}
{(2 \sqrt{\pi t})^3}
\Big(\frac{k}{2t} \sum\limits_{j=1}^3 e_j x_j + \mathfrak{f}(\frac{3}{2t} 
+ \frac{k|{\bf x}|^2}{4t^2})+k \mathfrak{f}^{\dagger}\Big).
$$
Suppose that $G$ is in general a space-time varying bounded Lipschitz domain $G \subset \R^3 \times \R^+$.  
In what follows $W_2^{k,l}(G)$ denotes the parabolic Sobolev spaces of $L_2(G)$ where $k$ is the regularity parameter with respect to ${\bf x}$ and $l$ the regularity parameter with respect to $t$. 
For our needs we recall, cf. e.g. \cite{Be,CK2,CV2009} 
\begin{theorem}(Borel-Pompeiu integral formula)\label{bp} 
Let $G \subset \mathbb{R}^3 \times \R^+$ be a bounded or unbounded Lipschitz domain with a
strongly Lipschitz boundary $\Gamma=\partial D$. Then for all $u
\in W_{2}^{1,1}(G)$ 
$$
\int\limits_{\Gamma} E({\bf x}-{\bf y},t-t_0;k) d\sigma_{{\bf x},t} u({\bf x},t) = u({\bf y},t_0) + \int\limits_{G} E({\bf x}-{\bf y},t-t_0;k) D_{{\bf x},t;k}^{+}(u({\bf x},t)) dV dt,
$$
where $d\sigma_{{\bf x},t} = D_{{\bf x},t} \rfloor dV dt$. The differential form $d\sigma_{{\bf x},t} = D_{{\bf x},t} \rfloor dV dt$ is the contraction 
of the operator $ D_{{\bf x},t}$ with the volume element $dV dt$.  
\end{theorem}
For $g \in $ Ker $D_{{\bf x,t;k}}^{+}$ one obtains the following version of Cauchy's integral formula for left parabolic $k$-monogenic functions in the form 
$$
\int\limits_{\Gamma} E({\bf x}-{\bf y},t-t_0;k) d\sigma_{{\bf x},t} u({\bf x},t) = u({\bf y},t_0).
$$
Again, following the above cited works, one can introduce the parabolic Teodorescu transform and the Cauchy transform by 
\begin{eqnarray*}
T_G u({\bf y},t_0) &=& \int_G E({\bf x}-{\bf y},t-t_0;k) u({\bf x},t) dV dt\\
F_{\Gamma} u({\bf y},t_0) &=& \int_{\Gamma} E({\bf x}-{\bf y},t-t_0;k) d\sigma_{{\bf x},t} u({\bf x},t). 
\end{eqnarray*}
\par\medskip\par
Analogously to the Euclidean case one can rewrite the Borel-Pompeiu formula in the form 
\begin{lemma} Let $u \in W_{2}^{1,0}(G)$. Then $T_G D_{{\bf x},t;k}^+u = u - F_{\Gamma} u$.   
\end{lemma} 
On the other hand one has $D_{{\bf x},t:k}^+T_G u = u$. So, the parabolic Teodorescu operator is the right inverse to the parabolic Dirac operator.
\par\medskip\par 
The following direct decomposition of the space $L_2(G)$ into the subspace of functions that are
square-integrable and left parabolic $k$-monogenic in the inside of $G$ and its
complement will be applied in this paper. 
\begin{theorem} (Hodge decomposition). Let $G \subseteq \mathbb{R}^3 \times \R^+$ be a bounded or unbounded Lipschitz domain.
Then $L_2(G) = B(G) \oplus D_{{\bf x},t;k}^+ {\stackrel{\circ}{W}}_{2}^{1,1}(G)$
where $B(G) := L_2(G) \cap$ Ker $D_{{\bf x},t;k}^+$ is the 
Bergman space of left parabolic $k$-monogenic functions, and where ${\stackrel{\circ}{W}}_{2}^{1,1}(G)$ is the subset of $W_{2}^{1,1}(G)$ with vanishing boundary data. 
\end{theorem} 
Proofs of the above statements can be found for example in \cite{Be,CK2,CV2009}. 
\par\medskip\par
In what follows 
${\bf P} : L_2(G) \rightarrow B(G)$ denotes the
orthogonal Bergman projection while ${\bf Q} : L_2(G) \rightarrow
D_{{\bf x},t}^+ {\stackrel{\circ}{W}}_{2}^{1,1} (G)$ stands for the projection
into the complementary space in all that follows. One has ${\bf Q}
= {\bf I} - {\bf P}$, where ${\bf I}$ stands for the identity
operator.



\section{The incompressible in-stationary MHD equations revisited in the quaternionic calculus}

In the classical vector analysis calculus the in-stationary viscous incompressible MHD equations have the form   
\begin{eqnarray}
-\frac{1}{Re}\Delta {\bf u}  + \frac{\partial {\bf u}}{\partial t} +  ({\bf u} \;{\rm grad})\; {\bf u} + {\rm grad}\; p  &=& \frac{1}{\mu_0} {\rm rot} {\bf B} \times {\bf B} \;\;{\rm in}\;G\\
-\frac{1}{Rm}\Delta {\bf B} + \frac{\partial {\bf B}}{\partial t} - ({\bf u} \;{\rm grad})\; {\bf B} + ({\bf B} \;{\rm grad}) {\bf u} &=& 0\;\;{\rm in}\; G \\
{\rm div}\; {\bf u} &=& 0 \;\;{\rm in}\;G\\
{\rm div}\; {\bf B} &=& 0 \;\;{\rm in}\;G\\
{\bf u}={\bf 0},\; {\bf B} &=& {\bf h} \;\;{\rm at}\;\partial G
\end{eqnarray}
with given boundary data $\left. {\bf u} \right|_{\partial G}= {\bf g}={\bf 0}$ and $\left. {\bf B} \right|_{\partial G}={\bf h}$. To apply the quaternionic integral operator calculus to solve these equations we first express this system in the quaternionic language. 
\par\medskip\par
First we recall that we have for a time independent quaternionic function
$f: \R^4 \to \R^4,$ where $(x_0+{\bf x}) \to f(x_0+{\bf x}) = f_0(x_0+{\bf x}) + {\bf f}(x_0+{\bf x}),$ the relation 
$
{\cal{D}} f = {\rm grad}\;f_0 + {\rm rot}\;{\bf f} - {\rm div}\; {\bf f}.
$ 
Here $f_0 = \Re(f)$ is the scalar part of $f$ while ${\bf f} = \Vec(f) \in \R^3$ represents the vectorial part of $f,$ and $
{\cal{D}}:= \sum_{i=0}^3 e_i \frac{\partial }{\partial x_i}
$ 
is the quaternionic Cauchy-Riemann operator. Its vector part, denoted by ${\bf D}$, is the three dimensional 
Euclidean Dirac operator introduced in the previous section. In the case where ${\bf f}$ is a vector
valued function, i.e. a function defined in an open subset of 
$\R^3$ with values in $\R^3$ we have
$
{\bf D} {\bf f} = {\rm rot}\;{\bf f} - {\rm div}\; {\bf f} 
$. 
If $p$ is a scalar valued function defined in an open subset of $\R^3$, then we have ${\bf D} p = {\rm grad}\;p$.
\par\medskip\par
When applying these rules to the magnetic vector field ${\bf B} 
\in \R^3$ we obtain that  ${\bf D} {\bf B} = {\rm rot}\;{\bf B} -
{\rm div}\; {\bf B}$. In view of equation (4) which expresses that
there are no magnetic monopoles, this equation reduces to
${\bf D} {\bf B} = {\rm rot}\;{\bf B}$. Furthermore, we can express $({\bf D}{\bf B}) \times {\bf B} = \Vec(({\bf D}{\bf B})\cdot {\bf B})$ in terms of the quaternionic product $\cdot$. The divergence of an $\R^3$-valued vector field ${\bf f}$ can be expressed as div ${\bf f} = \Re({\bf D}{\bf f}$).  The threedimensional Euclidean Laplacian $\Delta=\sum_{i=1}^3 \frac{\partial^2}{\partial x_i^2}$ can be expressed in terms of the Dirac operator as $\Delta=-{\bf D}^2$, applying the rule $e_i^2=-1$ for all $i=1,2,3$. 
\par\medskip\par
Let us next assume that our functions are also dependent on the time variable $t$. Applying the formulas from the preceding section allow us to express the entities $-\frac{1}{Re}\Delta {\bf u}  + \frac{\partial {\bf u}}{\partial t}$ and $-\frac{1}{Rm}\Delta {\bf B}  + \frac{\partial {\bf B}}{\partial t}$ in the form 
\begin{eqnarray*}
-\frac{1}{Re}\Delta {\bf u}  + \frac{\partial {\bf u}}{\partial t} &= & (D_{{\bf x},t,Re}^+)^2 {\bf u}\\
-\frac{1}{Rm}\Delta {\bf B}  + \frac{\partial {\bf B}}{\partial t} &=&(D_{{\bf x},t,Rm}^+)^2 {\bf B}.
\end{eqnarray*}
with
\begin{eqnarray*}
D_{{\bf x},t,Re}^+{\bf u} &=&\frac{1}{\sqrt{Re}}{\bf D}{\bf u}+\mathfrak{f}\partial_t {\bf u}+\mathfrak{f}^\dagger {\bf u}\\
D_{{\bf x},t,Rm}^+ {\bf B} &=&\frac{1}{\sqrt{Rm}}{\bf D} {\bf B}+\mathfrak{f}\partial_t  {\bf B}+\mathfrak{f}^\dagger  {\bf B}
\end{eqnarray*}

Thus, the system (6)-(10) can thus be reformulated in the quaternions in the following way: 
\begin{eqnarray}
(D_{{\bf x},t,Re}^+)^2 {\bf u} + \Re({\bf u} \;{\bf D})\; {\bf u} + {\bf D}\; p  &\!\!\!=\!\!& \frac{1}{\mu_0} \Vec(({\bf D}{\bf B})\cdot {\bf B}) \;\;{\rm in}\;G   \label{Eq:11} \\
(D_{{\bf B},t,Rm}^+)^2 {\bf B}  - \Re({\bf u} \;{\bf D})\; {\bf B} + \Re({\bf B} \;{\bf D}) {\bf u} &\!\!\!=\!\!& 0\;\;{\rm in}\; G  \label{Eq:12} \\
\Re({\bf D} {\bf u}) &\!\!\!=\!\!& 0 \;\;{\rm in}\;G\\
\Re({\bf D} {\bf B}) &\!\!\!=\!\!& 0 \;\;{\rm in}\;G\\
{\bf u}={\bf 0},\; {\bf B} &\!\!\!=\!\!& {\bf h} \;\;{\rm at}\;\partial G.
\end{eqnarray}
The aim is now to apply the previously introduced hypercomplex integral operators in order to get computation formulas for the magnetic field ${\bf B}$, the velocity ${\bf u}$, and the pressure $p$. 

We remark that whenever we fix the magnetic field ${\bf B}$ in the stationary version of Equation (\ref{Eq:11}) we obtain (in the weak sense) the pressure $p$ and the velocity $\bf u$, c.f. \cite{Zeidler}. In a similar way, given $(\bf u, p)$ in Equation (\ref{Eq:12}) we can recover the magnetic field ${\bf B}.$ Moreover, the solution for magnetic field is unique if the operator is hypoelliptic. These results hold for the in-stationary case.

\section{The MHD equations in the more general context of some conformally flat spin $3$-manifolds}

Due to the conformal invariance of the Dirac operator, the related quaternionic differential and integral operator calculus canonically provides a simple access to easily transfer the results and representation formulas summarized in the previous section to the context of addressing analogous boundary value problems within the more general context of conformally flat spin manifolds.  

\par\medskip\par

As a consequence of the famous Liouville thorem, in dimensions $n\ge 3$ conformally flat manifolds are explicitly only those that possess atlasses whose transition functions are M\"obius transformations, because these are the only conformal transformations in $\mathbb{R}^n$ whenever $n \ge 3$. 
The treatment with quaternions (or with Clifford numbers in general) allow us to represent M\"obius transformations in the  compact form $f(x) = (ax+b)(cx+d)^{-1}$ where $a,b,c,d$ are quaternions satisfying to certain constraints, cf. \cite{BCKR}. 

Already the classical paper \cite{Kuiper} mentions one possibility to construct a number of examples of conformally flat manifolds, namely by factoring out a subdomain ${\cal{U}} $ of $\R^{3}$ by a torsion-free subgroup $\Gamma$ of the group of M\"{o}bius transformations $\Gamma$, under the additional condition that the latter acts strongly discontinuously on ${\cal{U}} $. 

The topological quotient ${\cal{U}}  / \Gamma$ then is a conformally flat manifold. Of course, this construction just addresses a subclass of all conformally flat manifolds. However, this subclass can be characterized in an intrinsic way. As shown in \cite{Kuiper}, the class of conformally flat manifolds of the form $U/\Gamma$ are exactly those  for which the universal cover of this manifold admits a local conformal diffeomorphism into $S^3$ which is a covering map $\tilde{{\cal{U}} } \rightarrow {\cal{U}}  \subset S^3.$  

\par\medskip\par

The most popular examples are $3$-tori, cylinders, real projective (rotation invariant) space and the hyperbolic manifolds considered in \cite{BCKR} that arise by factoring upper half-spaces, cones or positivity domains by arithmetic subgroups  of higher dimensional generalizations of the modular or Fuchsian group.  \cite{BCKR}.  

\par\medskip\par

In order to generalize and to apply the representation formulas and the results that we obtained in the previous sections for the instationary MHD system to the context of analogous instationary boundary value problems on conformally manifolds we only need to introduce the properly adapted analogues of the parabolic Dirac operator as well as the other hypercomplex integral operators on these manifolds. From the geometric point of view one is particularly interested in those conformally flat manifolds that have a spin structure, that means those that admit the construction of at least one spinor bundle over such a manifold. In many cases one gets more than just one spin structure which leads to the consideration of (several) spinor sections, in our case quaternionic spinor sections. For the geometric background we refer to \cite{LM}.  

\par\medskip\par

We explain the method at the simplest non-trivial example dealing with conformally flat spin $1$,$2$-cylinders and $3$-tori with inequivalent spinor bundles. This special example  illustrates in a nice way how one can transfer the results and construction method to other examples of conformally flat (spin) manifolds that again are constructed by factoring out a connected domain by a discrete arithmetic  group of some higher dimensional modular groups, such as those roughly outlined above. 

\par\medskip\par
  
For the sake of simplicity, let $\Omega_3:= \Z e_1 + \Z e_2 + \Z e_3$ be the orthonormal lattice in $\mathbb{R}^3$. Then the topological quotient space $\R^3/ \Omega_3$ represents a $3$-dimensional conformally flat compact torus denoted by $T_3$, over which one can construct exactly eight different  conformally inequivalent spinor bundles over $T_3$. With the additional time coordinate $t > 0$, this leads to the consideration of a toroidal time half-cylinder of the form $\Omega_3 \times [0,\infty)$ which then represents a non-compact manifold with boundary in upper half space of $\mathbb{R}^4 \stackrel{\sim}{=} \mathbb{H}$, $t>0$, denoted by $\mathbb{H}^+$. The invariance group is an abelian subgroup of the hypercomplex modular group $SL(2,\mathbb{H}^+)$ just acting on the space coordinates. More generally, we can also factor out sublattices of the form $\Omega_p := \Z e_1 + \cdots + \Z e_p$ where $1 \le p \le 3$. The topological quotients $\R^3 / \Omega_p$ are $1$-resp. $2$-cylinders in the cases $p=1$ and $p=2$ respectively, having infinite extensions also in $x_3$- (resp. also in the $x_2$-) coordinate direction.  

We recall that in general different spin structures on a spin manifold $M$ are detected by the number of distinct homomorphisms from the fundamental group $\Pi_{1}(M)$ to the group ${\Z}_{2}=\{0,1\}$. In the case of the 3-torus we have $\Pi_{1}(T_3)={\Z}^{3}$. There are two homomorphisms of ${\Z}$ to ${\Z}_{2}$. The first one is $\theta_{1}:{\Z}\rightarrow {\Z}_{2}:\theta_{1}(n)=0$ mod $2$ while the second one is the homomorphism $\theta_{2}:{\Z}\rightarrow {\Z}_{2}:\theta_{2}(n)=1$ mod $2$. Consequently there are $2^{3}$ distinct spin structures on $T_3$, or more generally, $2^p$ different spin structures on $T_p$ with $p \le 3$.  
\par\medskip\par
For the sake of generality, in what follows let $p \in \{1,2,3\}$. It is very easy to construct all conformally inequivalent different spinor bundles over $T_p$. To describe them let $l$ be an integer in the set $\{1,2,3\}$, and consider the sublattice ${\Z}^{l}={\Z}e_{1}+\ldots+{\Z}e_{l}$  where$(0 \le l \le p)$. For $l=0$ we put ${\Z}^{0}:=\emptyset$. 
There is also the remainder lattice  ${\Z}^{p-l}={\Z}e_{l+1}+\ldots+{\Z}e_{p}$. In this case ${\Z}^{p}=\{\underline{m}+\underline{n}:\underline{m}\in {\Z}^{l}$ and $\underline{n}\in {\Z}^{p-l}\}$.  Let us now assume that $\underline{m}=m_{1}e_{1}+\ldots+m_{l}e_{l}$. We identity $({\bf x},X)$ with $({\bf x}+\underline{m}+\underline{n},(-1)^{m_{1}+\ldots+m_{l}}X)$ where ${\bf x}\in \R^{3}$ and $X\in \mathbb{H}$. This identification gives rise to a quaternionic spinor bundle $E^{(l)}$ over $T_p$.

\par\medskip\par

Clearly, $\R^3$ is the universal covering space of $T_p$. Thus, there is a well-defined
projection map ${\cal{P}}: \R^3 \times \mathbb{R}^{+} \to T_p \times \mathbb{R}^{+}$, by identifying $({\bf x}+\omega,t)$ with all equivalent points of the form $({\bf x} \;{\rm mod}\; \Omega_p,t)$.

As explained for example in \cite{BCKR} every $p$-fold periodic resp. anti-periodic open set ${\cal{U}} \subset \R^3$
and every $p$-fold periodic resp. anti-periodic section $f: {\cal{U}}' \times [0,\infty) \to E^{(l)}$, which satisfies $f({\bf x},t) = (-1)^{m_1+\cdots+m_l}({\bf x} + \omega,t)$ for all $\omega \in \Z^{l}\oplus \Z^{p-l}$, descends to a well-defined open set $U':={\cal{P}}({\cal{U}}) \times[0,\infty) \subset T_p\times [0,\infty) $ (associated with that particularly chosen spinor bundle) and a well-defined spinor section 
$f':={\cal{P}}(f): U' \subset T_p \times [0,\infty) \to E^{(l)} \subset \mathbb{H}$, respectively. 

The projection ${\cal{P}}: \R^3 \times [0,\infty) \to T_p \times[0,\infty)$ induces well-defined cylindrical resp. toroidal modified parabolic Dirac operators on $T_p \times \R^+$ by ${\cal{P}}(D_{{\bf x},t,k}^{\pm}) =: {\cal {D}}_{{\bf x},t,k}^{\pm}$ acting on spinor sections of $T_p \times \R^+$. Sections defined on open sets $U$ of $T_p \times \R^+$ are called cylindrical resp. toroidal $k$-left parabolic monogenic if ${\cal {D}}_{{\bf x},t,k}^{\pm} s = 0$ holds in $U$. By ${\tilde{D}}:={\cal{P}}({\bf D})$ we denote the projection of the time independent Euclidean Dirac operator down to the cylinder resp. torus $T_p$. 

\par\medskip\par

We denote the projections of the $p$-fold (anti-)periodization of the function $E({\bf x},t;k)$ by 
$$ 
{\cal{E}}({\bf x},t;k) := \sum\limits_{\omega
\in \Z^p\oplus \Z^{p-l}}(-1)^{m_1+\cdots+m_l} E({\bf x}+\omega,t;k).
$$ 
This generalized parabolic monogenic Eisenstein type series provides us with the fundamental section to the cylindrical resp. toroidal parabolic modified Dirac operator 
${\cal {D}}_{{\bf x},t,k}^{\pm}$ acting on the corresponding spinor bundle of the space cylinder resp. space torus $T_p$. Indeed, the function ${\cal{E}}({\bf x},t;k)$ can be regarded as the canonical generalization of the classical elliptic Weierstra{\ss} $\wp$-function to the context of the modified Dirac operator $D_{{\bf x},t,k}^{+}$ in three space variables $x_1,x_2,x_3$ and the positive time variable $t > 0$. 

\par\medskip\par

To show that  ${\cal{E}}({\bf x},t;k)$ is well-defined parabolic monogenic spinor section on the manifold $T_p \times[0,\infty)$, we have to show that this series actually converges. The regularity behavior then is guaranteed by the application of the Weierstra{\ss} convergence theorem. 
\begin{theorem} Let $1 \le p \le 3$. Then the function series
$$ 
{\cal{E}}({\bf x},t;k) = \sum\limits_{\omega \in \Z^p\oplus \Z^{p-l}}(-1)^{m_1+\cdots+m_l} E({\bf x}+\omega,t;k)
$$ 
converges uniformally on any compact subset of $\R^3 \times \R^+$. 
\end{theorem}
{\bf Proof}: The simplest way to prove the convergence is to decompose the full lattice $\mathbb{Z}^p$ into the 
the following particular union of lattice points $\Omega = \bigcup_{m=0}^{+\infty} \Omega_m$ where
$$\Omega_m := \{ \omega \in \Z^p \mid |\omega|_{max} = m\}.$$ 
Next one defines  
$$
L_m := \{
\omega \in \Z^p \mid |\omega|_{max} \le m\}.
$$
The subset $L_m$ contains exactly $(2m+1)^p$ points. Hence, the
cardinality of $\Omega_m$ precisely is $\sharp \Omega_m = (2m+1)^p -
(2m-1)^p$. Notice that this particular construction admits that Euclidean distance between the set $\Omega_{m+1}$
and the $\Omega_m$ is exactly $d_m := dist_2(\Omega_{m+1},\Omega_m)
= 1$. This is the motivation for this particular decomposition. 

\par\medskip\par

Next, as a standard calculus argument one fixes a compact subset ${\cal{K}} \subset \R^3$ and one considers $t > 0$ as an arbitrary but fixed value. 
Then there exists a $r \in \R$ such that all ${\bf x} \in
{\cal{K}}$ satisfy $|{\bf x}|_{max} \le |{\bf x}|_2 < r$.

Let ${\bf x} \in {\cal{K}}$.  For the convergence it suffice to consider those points with $|\omega|_{max} \ge [r]+1$.

As a consequence of the standard argumentation $$|{\bf x} + \omega|_2 \ge |\omega|_2 - |{\bf x}|_2 \ge
|\omega|_{max}-|{\bf x}|_2 = m - |{\bf x}|_2 \ge m - r$$ one may arrive at
\begin{eqnarray*}
& & \sum\limits_{m=[r]+1}^{+\infty} \sum\limits_{\omega \in \Omega_m} |E({\bf x},t;k)({\bf x}+\omega)|_2\\
& \le & \frac{k}{(2 \sqrt{\pi t})^3} \sum\limits_{m=[r]+1}^{+\infty} \sum\limits_{\omega \in \Omega_m} \exp(-k|{\bf x}+\omega|_2/4t)\Big(\frac{k}{2t} |{\bf x} + \omega|_2 + \mathfrak{f}(\frac{3}{2t} 
+ \frac{k|{\bf x}+\omega|_2^2}{4t^2})+k \mathfrak{f}^{\dagger}\Big)\\
& \le & \frac{k}{(2 \sqrt{\pi t})^3} \sum\limits_{m=[r]+1}^{+\infty}\Big( [(2m+1)^p - (2m-1)^p] \big(
\frac{k(r+m)}{2t} +\mathfrak{f}(\frac{3}{2t} + \frac{k(r+m)^2}{4t^2}) + k \mathfrak{f}^{\dagger}\big)\\
& & \quad\quad\quad\quad\quad\quad\quad\quad \times \;\; \exp(\frac{-k(m-r)^2}{4t})\Big),
\end{eqnarray*}
in view of $m - r \ge [r]+1-r > 0$. This sum is absolutely uniformly convergent because of the exponential decreasing term which dominates the polynomial expressions in $m$.  
Due to the absolute convergence, the series    
$${\cal{E}}({\bf x},t;k) :=
\sum\limits_{\omega \in \Z^l \oplus \Z^{p-l}} (-1)^{m_1+\cdots+m_l} E({\bf x}+\omega,t;k),$$
which can be can be rearranged in the requested form
$$
{\cal{E}}({\bf x},t;k) :=
\sum\limits_{m=0}^{+\infty}\sum\limits_{\omega \in \Omega_m} (-1)^{m_1+\cdots+m_l}
E({\bf x}+\omega,t;k),
$$
converges normally on $\R^3 \times \R^+$. Since
$E({\bf x}+\omega,t;k)$ belongs to Ker $D_{{\bf x},t,k}^{+}$ in each
$({\bf x},t) \in \R^3 \times \R^+$  the series ${\cal{E}}({\bf x},t;k)$ satisfies
$D_{{\bf x},t,k}^{+} {\cal{E}}({\bf x},t;k)  = 0$ in 
each ${\bf x} \in \R^3 \times \R^+$, which, as mentioned previously, follows from the classical standard Weierstra{\ss} convergence argument. \hfill $\blacksquare$
\par\medskip\par
Obviously, by a direct rearrangement argument, one obtains that 
$$
{\cal{E}}({\bf x},t;k) =(-1)^{m_1+\cdots+m_l}  {\cal{E}}({\bf x}+\omega,t;k)\;\;\;\forall \omega \in \Omega
$$
which shows that the projection of this kernel correctly descends to a section with values in the spinor bundle $E^{(l)}$. The projection ${\cal{P}}({\cal{E}}({\bf x},t;k))$ denoted by $\tilde{{\cal{E}}}({\bf x},t;k)$ is the fundamental section of the cylindrical (resp. toroidal) modified parabolic Dirac operator ${\tilde{D}}_{{\bf x},t,k}^{+}$. For a time-varying Lipschitz domain $G \subset T_3 \times \R^+$ with a strongly Lipschitz boundary $\Gamma$ we can now proceed to define, similarly to our description in the previous sections, the canonical analogue of the Teodorescu and of the Cauchy-Bitzadse transform for toroidal $k$-monogenic parabolic quaternionic spinor valued sections by 
\begin{eqnarray*}
\tilde{T}_G u({\bf y},t_0) &=& \int_G \tilde{{\cal{E}}}({\bf x}-{\bf y},t-t_0;k) u({\bf x},t) dV dt\\
\tilde{F}_{\Gamma} u({\bf y},t_0) &=& \int_{\Gamma} \tilde{{\cal{E}}}({\bf x}-{\bf y},t-t_0;k) d\sigma_{{\bf x},t} u({\bf x},t).
\end{eqnarray*}
To transfer the integral operator calculus from the flat Euclidean space setting to our setting we introduce the following norms on the manifolds and on the sections with values in the associated spinor bundles. Let $({\bf x}',t)$ be an arbitrary point on $T_p \times [0,\infty)$. Then we put for $1 \le q \le \infty$:
$$
\|({\bf x}',t)\|_{T_p,q} := \|{\cal{P}}^{-1}({\bf x}',t)\|_q := \min_{\omega \in \Omega_p} \|({\bf x}+\omega,t)\|_q
$$
where $\|\cdot\|_q$ is the usual $q$-norm on $\mathbb{R}^3 \times [0,\infty)$. 

Next we define the $L_q$-norm on an arbitrary quaternionic spinor section $f' :U':= {\cal{U}} \times[0,\infty) \subset T_p \times[0,\infty) \to E^{(l)} \subset \mathbb{H}$ with values in one of the previously described spinor bundles $E^{(l)}$ by:
$$
\|f'\|_{L_q(U')} := \sqrt[q]{\int\limits_{U} \min_{\omega \in \Omega_p} \{\|{\cal{P}}^{-1} f(({\bf x}+\omega,t))  \|^q  \}     d{\bf x} dt}
$$
Similarly, for $q < \infty$ we may introduce the adequate Sobolev spaces of derivative degree up to a fixed $k \ge 1$ by:
$$
\|f'\|_{W^{k}_q(U')} := \Bigg( \|f\|^q_{L^2(U')} + \sum\limits_{0 < \|\alpha\|+\beta \le k} \Bigg\|\frac{\partial^{|\alpha|+\beta}}{\partial {\bf x}^{\alpha} \partial t^{\beta}} \Bigg\|^q_{L^2(U')}  \Bigg)^{1/q}.
$$
An important property is the $L_1$-boundedness of the cylindrical (toroidal) fundamental solution $\tilde{{\cal{E}}}({\bf x}',t)$ in the norm $\|\cdot\|_{L_1}$. To justify this we note that in view of using the particular definition of the norm $\|\cdot\|_{T_p,1}$ we obtain: 
\begin{eqnarray*}
\|\tilde{{\cal{E}}}\|_{L_1} &=& \int\limits_{U'}\|\tilde{{\cal{E}}}({\bf x}',t)\|_{T_p,1} d{\bf x}'dt \\
&=& \int\limits_{U} \min_{\omega \in \Omega_p}\|E({\bf x}+\omega,t)  \|_1 d{\bf x} dt < \infty,
\end{eqnarray*}
since the fundamental solution $E$ is an $L_1$-function over any bounded domain $U$ in $\mathbb{R}^3 \times \mathbb{R}^{+}$ according to \cite{CK2}. This allows us directly to establish
\begin{proposition}
Let $1 \le q< \infty$. Let $G'\subset T_p \times[0,\infty)$ be a bounded domain. Then the operator $\tilde{T}_{G'}$ is bounded from $L_q(G')$ to $L_q(G')$. 
\end{proposition}
\begin{proof}
In view of Young's inequality we have
$$
\|\tilde{T}_{G'} g\|_{L_q(G')} = \| \tilde{{\cal{E}}} * g\|_{L_q(G)'} \le \| \tilde{{\cal{E}}}\|_{L_1(G')} \cdot \| g \|_{L_q(G')}.
$$
Since $\|\tilde{{\cal{E}}}\|_{L_1(G')}$ is a finite expression whenever $G'$ is bounded, as shown previously, we obtain the $L_q$-boundedness of $\tilde{T}_{G'}$.
\end{proof}
As furthermore shown in \cite{CK2} also the partial derivatives of $E({\bf x},t)$ are $L_1$-bounded under the condition that $G$ is a bounded domain, we directly obtain by a similar argument the following
\begin{proposition}
Let $1 \le q< \infty$. Let $G'\subset T_p \times[0,\infty)$ be a bounded domain. Then the partial derivatives of the operator $\tilde{T}_{G'}$ with respect to $x_k$ ($k=1,2,3$) satisfy the mapping property:
$$
\partial_{x_k} (\tilde{T}_{G'} g) : L_q(G') \rightarrow L_q(G'),\; k=1,2,3
$$
and are bounded.
\end{proposition}
To the proof one again only needs to apply Young's inequality leading to
$$
\|\partial_{x_k} (\tilde{T}_{G'} g)\|_{L_q(G')} = \|(\partial_{x_k} \tilde{{\cal{E}}}) * g\|_{L_q(G)'} \le \| \partial_{x_k} \tilde{{\cal{E}}}\|_{L_1(G')} \cdot \| g \|_{L_q(G')}.
$$ 
As a direct consequence of these two propositions we may now establish the important result
\begin{theorem}
Let $p\in \{1,2,3\}$, $1 \le q< \infty$ and let $k \in \mathbb{N}$. Let $G'$ be a bounded domain in the time $p$-cylinder (torus) $T_p \times[0,\infty)$. Then the operator $\tilde{T}_{G'}:L_q(G') \to W^k_q(G')$ is continuous. 
\end{theorem}
This property together with the Borel-Pompeiu formula presented in Section~2 also implies that the operator
$$
\tilde{F}_{\Gamma}:W^{k-1/q}_q(\Gamma) \to W^{k}_q(G')
$$
is continuous. 

To complete the quaternionic integral calculus toolkit, the associated Bergman projection can be introduced by 
$$
\tilde{{\bf P}} = \tilde{F}_{\Gamma}(tr_{\Gamma} \tilde{T}_G \tilde{F}_{\Gamma})^{-1} tr_{\Gamma}
\tilde{T}_G.
$$
and $\tilde{{\bf Q}} := \tilde{{\bf I}} - \tilde{{\bf P}}$. 
\par\medskip\par
Now, adapting from \cite{CK2009} we obtain a direct analogy of Theorem~1, Lemma~1 and Lemma~2 on these conformally flat time cylinders rep. time tori using these time cylindrical (toroidal) versions $\tilde{T}_G, \tilde{F}_{\Gamma}$ and $\tilde{{\bf P}}$ of operators introduced in Section~2. 
Suppose next that we have to solve an MHD problem of the form (1)-(5) within a Lipschitz domain $G \subset T_3 \times \R^+$ with values in the spinor bundle $E^{(l)} \times \R^+$. Then, imposing certain regularity conditions, which will be discussed in very detail in our future work, we can compute its solutions by simply applying the following adapted iterative algorithm 
\begin{eqnarray*}
{\bf u}_n &=& \frac{Re}{\mu_0} \tilde{T}_G \tilde{{\bf Q}} \tilde{T}_G \Big[\Vec(({\tilde{D}}{\bf B}_{n-1})\cdot {\bf B}_{n-1})-\Re({\bf u}_{n-1}{\tilde {D}}){\bf u}_{n-1}\Big]\\ & & - Re^2 \tilde{T}_G \tilde{{\bf Q}}\tilde{T}_G {\tilde{D}} p_n\\
\Re(\tilde{{\bf Q}}\tilde{T}_G {\tilde{D}} p_n)&=&\frac{1}{\mu_0}\Re\Big[\tilde{{\bf Q}} \tilde{T}_G \Vec(({\tilde{D}}{\bf B}_{n-1})\cdot {\bf B}_{n-1}) - \Re({\bf u}_{n-1}{\tilde{D}}){\bf u}_{n-1}\Big]\\
{\bf B}_n &=& Rm^2 \tilde{T}_G \tilde{{\bf Q}} \tilde{T}_G \Big[\Re({\bf B}_n{\tilde{D}}){\bf u}_n - \Re({\bf u}_n{\tilde{D}}){\bf B}_n \Big].\\
{{\bf B}_n}^{(i)} &=& Rm^2 \tilde{T}_G \tilde{{\bf Q}} \tilde{T}_G \Big[\Re({{\bf B}_n}^{(i-1)}{\tilde{D}}){\bf u}_n - \Re({\bf u}_n{\tilde{D}}){{\bf B}_n}^{(i-1)} \Big]
\end{eqnarray*}
\par\medskip\par
Again, in our future work, we will address a number of concrete existence and uniqueness criteria for the solutions computed by this fixed point algorithm involving some a priori estimate conditions. 

\par\medskip\par
Anyway, it is now clear how this approach even carries over to more general conformally flat spin manifolds that arise by factoring out a simply connected domain $U$ by a discrete Kleinian group $\Gamma$. The Cauchy-kernel is constructed by the projection of the $\Gamma$-periodization (involving eventually automorphy factors like in \cite{BCKR}) of the fundamental solution $E({\bf x};t;k)$. With this fundamental solution we construct the corresponding integral operators on the manifold. In terms of these integral operators we can express the solutions of the corresponding MHD boundary value problem on these manifolds, simply by replacing the usual hypercomplex integral operators by its adequate analogies on the manifold. In this framework, of course one has to introduce the adequated norms and to consider the adequated function spaces accordingly.  

This again underlines the highly universal character of our approach to treat the MHD equations but also many other complicated elliptic, parabolic, hypoelliptic and hyperbolic PDE systems with the quaternionic operator calculus using Dirac operators. Furthermore, the representation formulas and results also carry directly over to the $n$-dimensional case in which one simply replaces the corresponding quaternionic operators by Clifford algebra valued operators, such as suggested in \cite{CK2,CK2009}. 
\par\medskip\par
To round off we establish a further result on the invariance behavior of the kernel functions  under rotations of $S^3$ applied to the spatial coordinates. More precisely, we have:
\begin{theorem} Let $a \in S^{3} := \{{\bf x} \in \mathbb{R}^3 \mid \|{\bf x}\|=1\}$. Then the Cauchy kernel of the parabolic Dirac operator satisfies the invariance property $\overline{a} E(a {\bf x} \overline{a},t;k) a = E({\bf x},t;k)$ for all $a \in S^3$. 
\end{theorem}
\begin{proof}
Let us consider the expression:
\begin{eqnarray*}
\overline{a} E(a {\bf x} \overline{a},t;k) a &=& \overline{a} \Bigg(  
\frac{H(t) \exp(-\frac{|a{\bf x}\overline{a}|^2}{4t})}
{(2 \sqrt{\pi t})^3}
\Big(\frac{1}{2t} a{\bf x}\overline{a} + \mathfrak{f}(\frac{3}{2t} 
+ \frac{|a{\bf x}\overline{a}|^2}{4t^2})+ \mathfrak{f}^{\dagger}\Big)   
\Bigg) a\\
&=&   
\frac{H(t) \exp(-\frac{|{\bf x}|^2}{4t})}
{(2 \sqrt{\pi t})^3}
\Big(\frac{1}{2t} \overline{a}a{\bf x}\overline{a}a + \overline{a}\mathfrak{f}(\frac{3}{2t} 
+ \frac{|{\bf x}|^2}{4t^2})a+ \overline{a}\mathfrak{f}^{\dagger}a\Big) \\
&=& E({\bf x},t;k)
\end{eqnarray*}
where we applied the properties that $a \overline{a} = \|a\|^2=1$, $\overline{a} \mathfrak{f} a = \mathfrak{f}$ and 
$\overline{a} \mathfrak{f}^{\dagger} a = \mathfrak{f}^{\dagger}$.
\end{proof}
This property opens the door to treat a class of $S^3$-invariant manifolds. More precisely,   
by identifying all points of the form $(a {\bf x} \overline{a},t)$ with $({\bf x},t)$ we can construct a class of rotation invariant projective orbifolds which under certain constraints on ${\bf a}$ will be manifolds again. 

\par\medskip\par

Notice also the cylindrical and toroidal kernels ${\cal{E}}({\bf x}',t)$ exhibit this rotation invariance behavior. This is due to the fact that each single term in the series itself exhibits this rotation invariance property, so that the whole series turn out to have this property.  
\par\medskip\par
Moreover, this new identification can additionally be combined with the cylindrical (toroidal) translation invariance where one applies the identification of all $\Omega_p$-equivalent points. This gives rise to an identification of all points of the time cylinder (torus) $(a {\bf x}' \overline{a},t)$ with $({\bf x}',t)$. The associated orbifold resulting from this identification that has both a translation and a rotation invariant structure. In some dimensions we even obtain manifolds.

In the case where we restrict to those points from the unit sphere $a \in S^3$ such that there is a finite number $n \in \mathbb{N}$ with $a^n = 1$ which yields a finite cyclic group of rotations ${\cal{A}}:=\{a,a^2,\ldots,a^n\}$, then the corresponding Cauchy kernel can again be constructed by an Eisenstein type series. The latter then has the explicit form 
$$ 
{\cal{E}}_{\cal{A}}({\bf x},t;k) = \sum\limits_{a \in {\cal{A}}} \sum\limits_{\omega \in \Z^p\oplus \Z^{p-l}}(-1)^{m_1+\cdots+m_l} \overline{a}E(a{\bf x}\overline{a}+\omega,t;k)a
$$   
which then descends to a projective rotational variant of the cylinders / tori discussed previously. Since ${\cal{A}}$ only has a finite cardinality, the convergence of this series is guaranteed by the argument of Theorem 3. 

Once one has that the kernel function, one again can introduce the corresponding Teodorescu and Cauchy Bitzadse operators involving these explicit kernels in the same way as performed previously to also address the corresponding boundary value problems in these kinds of geometries introducing the norms properly. This once more underlines the geometric universality of our approach where we do nothing else than exploiting the conformal invariance of the Dirac operator. 

\section{Acknowledgements}  

The work of the third  author is supported by the project \textit{Neue funktionentheoretische Methoden f\"ur instation\"are PDE}, funded by Programme for Cooperation in Science between Portugal and Germany, DAAD-PPP Deutschland-Portugal, Ref: 57340281. The work of the first and second authors is supported via the project \textit{New Function Theoretical Methods in Computational Electrodynamics} approved under the agreement A\c c\~oes Integradas Luso-Alem\~as DAAD-CRUP,  ref. A-15/17, and by Portuguese funds through the CIDMA - Center for Research and Development in Mathematics and Applications, and the Portuguese Foundation for Science and Technology (``FCT--Funda\c{c}\~ao para a Ci\^encia e a Tecnologia''),  within project UID/MAT/0416/2013.

\end{document}